\documentclass[11pt]{article}
\renewcommand{\baselinestretch}{1.2}
\newcommand{\dated}{\mbox{} \hfill {\small [{\tt \today}]}} 
\usepackage{amsthm,enumerate}
\theoremstyle{plain}
\newtheorem{theorem}{Theorem}[section]
\newtheorem{lemma}[theorem]{Lemma}
\newtheorem{corollary}[theorem]{Corollary}
\newtheorem{proposition}[theorem]{Proposition}
\theoremstyle{definition}
\newtheorem{definition}[theorem]{Definition}
\theoremstyle{remark}
\newtheorem*{remark}{Remark}

\newtheorem*{rems}{Remarks}
\newtheorem*{exs}{Examples}

\newenvironment{items}{\begin{enumerate}[\rm (i)]}{\end{enumerate}}
\newenvironment{alphitems}{\begin{enumerate}[\rm (a)]}{\end{enumerate}}

 \usepackage{amsmath,amssymb,amsfonts,diagrams}
\newenvironment{keywords}{\noindent\small {\it Keywords\/}:}{\vskip 4pt}
\newenvironment{classification}{\noindent\small 2000 {\it Mathematics Subject
Classification\/}:}{\vskip 12pt}

\newcommand{\comps}{{\mathbb C}}

\newcommand{\ints}{{\mathbb Z}}
\newcommand{\posints}{{\mathbb N}}

\newcommand{\tensor}{\otimes}
\newcommand{\ttensor}{\tilde{\otimes}}
\newcommand{\Tensor}{\hat{\otimes}}

\newcommand{\wTensor}{\check{\otimes}}

\newcommand{\cstar}{{C^\ast}}

\newcommand{\id}{{\mathrm{id}}}

\newcommand{\Tr}{{\operatorname{Tr}}}

\newcommand{\A}{{\mathfrak A}}
\newcommand{\B}{{\mathfrak B}}
\newcommand{\Hilbert}{{\mathfrak H}}

\newcommand{\op}{{\mathrm{op}}}

\newcommand{\SL}{\operatorname{SL}}

\theoremstyle{plain}
\newtheorem*{taka}{Ozawa's Lemma}
\title{${\cal B}(\ell^p)$ is never amenable}
\author{\textit{Volker Runde}\thanks{Research supported by NSERC}}
\date{}
\begin{document}
\maketitle
\begin{abstract}
We show that, if $E$ is a Banach space with a basis satisfying a certain condition, then the Banach algebra $\ell^\infty({\cal K}(\ell^2 \oplus E))$ is not amenable; in particular, this is true for $E = \ell^p$ with $p \in (1,\infty)$. As a consequence, $\ell^\infty({\cal K}(E))$ is not amenable for any infinite-dimensional ${\cal L}^p$-space. This, in turn, entails the non-amenability of ${\cal B}(\ell^p(E))$ for any ${\cal L}^p$-space $E$, so that, in particular, ${\cal B}(\ell^p)$ and ${\cal B}(L^p[0,1])$ are not amenable.
\end{abstract}
\section*{Introduction}
The theory of amenable Banach algebras begins with B.\ E.\ Johnson's memoir \cite{Joh1}. The choice of terminology comes from \cite[Theorem 2.5]{Joh1}: a locally compact group is amenable (in the usual sense) if and only if its $L^1$-algebra satisfies a certain cohomological triviality condition, which is then used to define the class of amenable Banach algebras.
\par 
The notion of Banach algebraic amenability has turned out to be extremly fruitful ever since the publication of \cite{Joh1}. The class of amenable Banach algebras is small enough to allow for strong theorems of general nature, yet large enough to encompass a diverse collection of examples, such as the nuclear $\cstar$-algebras (see \cite[Chapter 6]{LoA} for an account), Banach algebras of compact operators on certain, well behaved Banach spaces (\cite{GJW}), and even radical Banach algebras (\cite{RunRad} and \cite{Rea0}).
\par 
The memoir \cite{Joh1} concludes with a list of suggestions for further research. One of them (\cite[10.4]{Joh1}) asks:
\begin{quote}
\textit{Is ${\cal B}(E)$---\emph{the Banach algebra of all bounded linear operators on a Banach space $E$}---ever amenable for any infinite-dimensional $E$?}
\end{quote}
\par
From a philosophical point of view, the answer to this question ought to be a clear ``no''. As for groups, amenability for Banach algebras can be viewed as a weak finiteness condition: amenable Banach algebras tend to be ``small''---wahatever that may mean precisely---, and ${\cal B}(E)$ simply feels too ``large'' to be amenable. It seems, however, as if Johnson's question has recently---somewhat surprisingly---found a positive answer: in \cite{AH}, S.\ A.\ Argyros and R.\ G.\ Haydon construct an infinite dimensional Banach space $E$ with few bounded linear operators, i.e., ${\cal B}(E) = {\cal K}(E) + \comps \, \id_E$ (with ${\cal K}(E)$ denoting the compact operators on $E$). As H.\ G.\ Dales pointed out to the author, $E$ has property $(\mathbb{A})$ introduced in \cite{GJW}, so that ${\cal K}(E)$ is an amenable Banach algebra for this space as is, consequently, ${\cal B}(E)$.
\par 
Still, infinite-dimensional Banach spaces $E$ with ${\cal B}(E)$ amenable ought to be the exception rather than the rule. Indeed, it follows from work by S.\ Wassermann (\cite{Was}) and the equivalence of amenability and nuclearity for $\cstar$-algebras that ${\cal B}(\ell^2)$ cannot be amenable. With $\ell^2$ being the ``best behaved'' of all $\ell^p$-spaces, one is led to expect that ${\cal B}(\ell^p)$ fails to be amenable for all $p \in [1,\infty]$. However, until recently it wasn't known for any $p \in [1,\infty]$ other than $2$ whether or not ${\cal B}(\ell^p)$ is amenable. The first to establish the non-amenability of ${\cal B}(\ell^p)$ for any $p \in [1,\infty] \setminus \{ 2 \}$ was C.\ J.\ Read in \cite{Rea}, where he showed that ${\cal B}(\ell^1)$ is not amenable. Subsequently, Read's proof was simplified by G.\ Pisier (\cite{Pis}). Eventually, N.\ Ozawa, simplified Pisier's argument even further and gave a proof that simultaneously establishes the non-amenability of ${\cal B}(\ell^p)$ for $p=1,2,\infty$ (\cite{Oza}).
\par 
In \cite{DR}, M.\ Daws and the author investigated the question of whether ${\cal B}(\ell^p)$ is amenable for $p \in (1,\infty) \setminus \{ 2 \}$. Among the results obtained in \cite{DR} are the following: If ${\cal B}(\ell^p)$ is amenable, then so are $\ell^\infty({\cal B}(\ell^p))$ and $\ell^\infty({\cal K}(\ell^p))$. The amenability of $\ell^\infty({\cal K}(\ell^p))$, in turn, forces $\ell^\infty({\cal K}(E))$ to be amenable for every infinite-dimensional ${\cal L}^p$-space in the sense of \cite{LP}; in particular, if ${\cal B}(\ell^p)$ is amenable, then so is $\ell^\infty({\cal K}(\ell^2 \oplus \ell^p))$.
\par 
This last implication is the starting point of this paper. Through a modification of Ozawa's approach from \cite{Oza}, we show that $\ell^\infty({\cal K}(\ell^2 \oplus E))$ is not amenable for certain Banach spaces $E$, including $E = \ell^p$ for all $p \in (1,\infty)$. As a consequence, ${\cal B}(\ell^p)$ cannot be amenable for such $p$ (and neither can ${\cal B}(L^p[0,1]))$. Together with the results from \cite{Rea} and \cite{Oza}, this proves that ${\cal B}(\ell^p)$ is not amenable for any $p \in [1,\infty]$.
\subsubsection*{Acknowledgment}
I am grateful to Gilles Pisier and Seytek Tabaldyev for pointing out to me that the original requirement of Theorem \ref{nonamthm} that $E$ have a shrinking basis was stronger than necessary. I owe special thanks to Seytek Tabaldyev also for discovering an error in an earlier version of the proof of that theorem.
\section{Amenable Banach algebras}
The original definition of an amenable Banach algebra from \cite{Joh1} is given in terms of first order Hochschild cohomology. An equivalent, but more intrinsic characterization---through approximate and virtual diagonals---was given soon thereafter in \cite{Joh2}. For the work done in this paper, however, yet another equivalent characterization of amenability due to A.\ Ya.\ Helemski\u{\i} turns out to be best suited (\cite[Theorem VII.2.20]{Hel}).
\par 
We denote the algebraic tensor product by $\tensor$ and use the symbol $\Tensor$ for the projective tensor product of Banach spaces. If $\A$ and $\B$ are Banach algebras, then so is $\A \Tensor \B$ in a canonical fashion. For any Banach algebra $\A$, we use $\A^\op$ for its \emph{opposite algebra}, i.e., the underlying Banach space is the same as for $\A$, but multiplication has been reversed. Multiplication in $\A$ induces a bounded linear map $\Delta \!: \A \Tensor \A \to \A$; it is immediate that $\ker \Delta$ is a left ideal in in $\A \Tensor \A^\op$.
\begin{definition} \label{amdef}
A Banach algebra $\A$ is said to be amenable if
\begin{alphitems}
\item $\A$ has a bounded approximate identity, and
\item the left ideal $\ker \Delta$ of $\A \Tensor \A^\op$ has a bounded right approximate identity.
\end{alphitems}
\end{definition}
\par 
Definition \ref{amdef} makes the proof of the following lemma, which we will require later on, particularly easy:
\begin{lemma} \label{amlem}
Let $\A$ be an amenable Banach algebra, and let $e \in \A$ be an idempotent. Then, for any $\epsilon > 0$ and any finite subset $F$ of $e\A e$, there are $a_1, b_1, \ldots, a_r, b_r \in \A$ such that
\begin{equation} \label{diag1}
  \sum_{k=1}^r a_k b_k = e
\end{equation}
and 
\begin{equation} \label{diag2}
  \left\| \sum_{k=1}^r x a_k \tensor b_k - a_k \tensor b_k x \right\|_{\A \Tensor \A} < \epsilon \qquad (x \in F).
\end{equation}
\end{lemma}
\begin{proof}
Let $\bullet$ denote the product in $\A \Tensor \A^\op$. Since $F \subset e \A e$, we have that $\{ x \tensor e - e \tensor x : x \in F \} \subset \ker \Delta$. By Definition \ref{amdef}(b), there is $\boldsymbol{r} \in \ker \Delta$ such that
\begin{equation} \label{rai}
  \| ( x \tensor e - e \tensor x) \bullet \boldsymbol{r} - (x \tensor e - e \tensor x) \| < \epsilon.
\end{equation}
\par 
Set $\boldsymbol{d} := e \tensor e - (e \tensor e) \bullet \boldsymbol{r}$. Without loss of generality, we can suppose that $\boldsymbol{d} \in \A \tensor \A$, i.e., there are $a_1, b_1, \ldots, a_r, b_r \in \A$ with $\boldsymbol{d} = \sum_{j=1}^r a_j \tensor b_j$. From the definition of $\boldsymbol{d}$, it is then immediate that (\ref{diag1}) holds while (\ref{rai}) translates into (\ref{diag2}).
\end{proof}
\begin{remark}
Since $\ker \Delta$ has \emph{bounded} right approximate identity, there is $C \geq 0$, depending only on $\A$, but not on $F$ or $\epsilon$, such that $a_1, b_1, \ldots, a_r , b_r \in \A$ satisfying (\ref{diag1}) and (\ref{diag2}) can be chosen such that $\left\| \sum_{k=1}^r a_k \tensor b_k \right\|_{\A \Tensor \A} \leq C$. We shall make no use of this, however.
\end{remark}
\section{Ozawa's proof revisited}
In \cite{Oza}, Ozawa presents a proof that simultaneously establishes the non-amenability of the Banach algebras ${\cal B}(\ell^p)$ for $p=1,2,\infty$ and $\text{$\ell^\infty$-}\bigoplus_{n=1}^\infty {\cal B}(\ell^p_N)$ for all $p \in [1,\infty]$. 
\par 
In this section, we recast the final step of his proof as a lemma, which doesn't make any reference to particular Banach algebras.
\par 
A pivotal r\^ole in Ozawa's argument is played by the fact that the group $\SL(3,\ints)$ has \emph{Kazhdan's property $(T)$} (\cite[Theorem 4.2.5]{BHV}) and thus, in particular, is finitely generated (\cite[Theorem 1.3.1]{BHV}) by $g_1, \ldots, g_m$, say. For the definition (and more) on Kazhdan's property $(T)$, we refer to \cite{BHV}.
\par 
What we require is the following consequence of $\SL(3,\ints)$ having $(T)$ (\cite[Theorem 3.1]{Oza}; compare also \cite[Proposition 1.1.9]{BHV}):
\begin{proposition} \label{T}
For any $g_1, \ldots, g_m$ generating $\SL(3,\ints)$, there is a constant $\kappa > 0$ such that, for any unitary representation $\pi$ of $\SL(3,\ints)$ on a Hilbert space $\Hilbert$ and any $\xi \in \Hilbert$, there is a $\pi$-invariant vector $\eta \in \Hilbert$, i.e., satisfying $\pi(g) \eta = \eta$ for all $g \in \SL(3,\ints)$, such that
\[
  \| \xi - \eta \| \leq \kappa \, \max_{j=1,\ldots,m} \| \pi(g_j)\xi-\xi\|.
\]
\end{proposition}
\par 
We briefly recall the setup laid out in \cite[Section 3]{Oza}, which we'll require both for the lemma at the end of this section and in the proof of Theorem \ref{nonamthm} below.
\par
Let $\mathbb P$ denote the set of all prime numbers, and let, for each $p \in \mathbb P$, the projective plane over the finite field $\ints / p \ints$ be denoted by $\Lambda_p$. Obviously, $\SL(3,\ints)$ acts on $\Lambda_p$ through matrix multiplication, which, in turn, induces a unitary representation $\pi_p$ of $\SL(3,\ints)$ on $\ell^2(\Lambda_p)$. The action of $\SL(3,\ints)$ on $\Lambda_p$ is $2$-transitive, i.e., the product action of $\SL(3,\ints)$ on $\Lambda_p \times \Lambda_p$ has exactly two orbits: the diagonal and its complement. Consequently, whenever $\boldsymbol{\eta} \in \ell^2(\Lambda_p \times \Lambda_p)$ is $\pi_p \tensor \pi_p$-invariant, there are  $\alpha, \beta \in \comps$ such that
\begin{equation} \label{invec1}
  \boldsymbol{\eta} = \alpha \boldsymbol{\zeta}_1 + \beta \boldsymbol{\zeta}_2
\end{equation}
with
\begin{equation} \label{invec2}
  \boldsymbol{\zeta}_1 = | \Lambda_p|^{-\frac{1}{2}} \sum_{\lambda \in \Lambda_p} e_\lambda \tensor e_\lambda
  \qquad\text{and} \qquad
  \boldsymbol{\zeta}_2 = | \Lambda_p |^{-1} \sum_{\lambda,\mu \in \Lambda_p} e_\lambda \tensor e_\mu.
\end{equation}
(Here, $e_\lambda$ for $\lambda \in \Lambda_p$ is the point mass at $\lambda$, as usual.) Finally, choose a subset $S_p$ of $\Lambda_p$ such that $|S_p| = \frac{|\Lambda_p| - 1}{2}$, and define an invertible isometry $v_p \in {\cal B}(\ell^2(\Lambda_p))$ through
\[
  v_p e_\lambda = \left\{ \begin{array}{rl} e_\lambda, & \lambda \in S_p, \\ -e_\lambda, & \lambda \notin S_p. \end{array} \right.
\]
With generators $g_1, \ldots, g_m$ of $\SL(3,\ints)$ fixed, we shall write $\pi_p(g_{m+1})$ instead of $v_p$ for notational convenience.
\par 
The following statement is proven (on \cite[pp.\ 561--562]{Oza}), albeit not explicitly stated in \cite[Section 5]{Oza}. For the reader's convenience, we include a proof.
\begin{taka}
It is impossible to find, for each $\epsilon > 0$, a number $r \in \posints$ with the following property: for each $p \in \mathbb P$, there are $\xi_{1,p}, \eta_{1,p}, \ldots, \xi_{r,p}, \eta_{r,p} \in \ell^2(\Lambda_p)$ such that $\sum_{k=1}^r \xi_{k,p} \tensor \eta_{k,p} \neq 0$ and
\begin{multline}
  \left\| \sum_{k=1}^r \xi_{j,p} \tensor \eta_{k,p} - (\pi_p(g_j) \tensor \pi_p(g_j))(\xi_{k,p} \tensor \eta_{k,p}) 
  \right\|_{\ell^2(\Lambda_p) \Tensor \ell^2(\Lambda_p)} \\ \leq \epsilon \left\| \sum_{k=1}^r \xi_{k,p} \tensor \eta_{k,p}
  \right\|_{\ell^2(\Lambda_p) \Tensor \ell^2(\Lambda_p)} \qquad (j=1, \ldots, m+1). \label{epsinv}
\end{multline}
\end{taka}
\begin{proof}
Assume towards a contradiction that, for each $\epsilon > 0$, there is $r \in \posints$ and, for each $p \in \mathbb P$, there are $\xi_{1,p}, \eta_{1,p}, \ldots, \xi_{r,p}, \eta_{r,p} \in \ell^2(\Lambda_p)$ with $\sum_{k=1}^r \xi_{k,p} \tensor \eta_{k,p} \neq 0$ such that (\ref{epsinv}) holds.
\par 
Suppose without loss of generality that $\sum_{k=1}^r \xi_{k,p} \tensor \eta_{k,p}$ has norm one in $\ell^2(\Lambda_p) \Tensor \ell^2(\Lambda_p)$, and let $\Phi \!: \ell^2 \Tensor \ell^2 \to \ell^2 \ttensor_2 \ell^2$ be the non-commutative Mazur map of \cite[Section 4]{Oza} (with $\ttensor_2$ denoting the Hilbert space tensor product). Then $\boldsymbol{\xi} := \Phi\left( \sum_{k=1}^r \xi_{k,p} \tensor \eta_{k,p} \right) \in \ell^2(\Lambda_p) \ttensor_2 \ell^2(\Lambda_2)$ is a unit vector satisfying
\[
  \| \boldsymbol{\xi} - (\pi_p(g_j) \tensor \pi_p(g_j))\boldsymbol{\xi} 
  \|_{\ell^2(\Lambda_p) \ttensor_2 \ell^2(\Lambda_2)} \leq \omega_\Phi(\epsilon) \qquad (j=1, \ldots, m+1).
\]
where $\omega_\Phi$ is the \emph{modulus of continuity} of $\Phi$ (\cite[p.\ 562]{Oza}). With $\kappa > 0$ as in Proposition \ref{T}, there is a $\pi_p \tensor \pi_p$-invariant vector $\boldsymbol{\eta} \in \ell^2(\Lambda_p) \ttensor_2 \ell^2(\Lambda_p)$ such that
\[
  \| \boldsymbol{\xi} - \boldsymbol{\eta} \|_{\ell^2(\Lambda_p) \ttensor_2 \ell^2(\Lambda_2)} \leq \kappa \, \omega_\Phi(\epsilon)
\]
which must be of the form (\ref{invec1}) with $\boldsymbol{\zeta}_1$ and $\boldsymbol{\zeta}_2$ as in (\ref{invec2}). From the definition of $\pi_p(g_{m+1})$, it follows that 
\[
  \begin{split}
  | \beta | & \leq | \beta | \left( 2 - 2 |\Lambda_p|^{-2} \right)^\frac{1}{2} \\
  & = | \beta | \| \boldsymbol{\zeta}_2 - 
    (\pi_p(g_{m+1}) \tensor \pi_p(g_{m+1}))\boldsymbol{\zeta}_2 \|_{\ell^2(\Lambda_p) \ttensor_2 \ell^2(\Lambda_2)} \\
  & = \| \boldsymbol{\eta} - 
    (\pi_p(g_{m+1}) \tensor \pi_p(g_{m+1}))\boldsymbol{\eta} \|_{\ell^2(\Lambda_p) \ttensor_2 \ell^2(\Lambda_2)} \\
  & \leq 2\kappa \, \omega_\Phi(\epsilon) + \| \boldsymbol{\xi} - 
    (\pi_p(g_{m+1}) \tensor \pi_p(g_{m+1}))\boldsymbol{\xi} \|_{\ell^2(\Lambda_p) \ttensor_2 \ell^2(\Lambda_2)} \\
  & \leq (2\kappa +1) \, \omega_\Phi(\epsilon)
  \end{split}
\]
and thus
\begin{equation} \label{666}
  \| \boldsymbol{\xi} - \alpha \boldsymbol{\zeta}_1 \|_{\ell^2(\Lambda_p) \ttensor_2 \ell^2(\Lambda_2)} 
  \leq \| \boldsymbol{\xi} - \boldsymbol{\eta} \|_{\ell^2(\Lambda_p) \ttensor_2 \ell^2(\Lambda_2)} + | \beta| 
  \leq (3\kappa +1) \, \omega_\Phi(\epsilon);
\end{equation}
in particular,
\begin{equation} \label{42}
  | \alpha | \geq 1 - (3\kappa +1) \, \omega_\Phi(\epsilon)
\end{equation}
holds.
\par 
On the other hand, we may view $\boldsymbol{\xi}$ as a Hilbert--Schmidt operator on $\ell^2(\Lambda_p)$ of rank at most $r$, so that
$| \Tr \, \boldsymbol{\xi} | \leq r^\frac{1}{2}$, where $\Tr$ is the canonical trace on ${\cal B}(\ell^2(\Lambda_p))$. The Cauchy--Schwarz inequality then yields
\[
  | \Lambda_p |^\frac{1}{2} \| \boldsymbol{\xi} - \alpha \boldsymbol{\zeta}_1 \|_{\ell^2(\Lambda_p) \ttensor_2 \ell^2(\Lambda_2)}
  \geq |\Tr (\alpha \boldsymbol{\zeta}_1 - \boldsymbol{\xi})| \geq | \alpha | | \Lambda_p|^\frac{1}{2} - r^\frac{1}{2}.
\]
Dividing by $|\Lambda_p|^\frac{1}{2}$ and taking (\ref{666}) and (\ref{42}) into account, we obtain
\[
  (3\kappa +1) \, \omega_\Phi(\epsilon) 
  \geq \| \boldsymbol{\xi} - \alpha \boldsymbol{\zeta}_1 \|_{\ell^2(\Lambda_p) \ttensor_2 \ell^2(\Lambda_2)}
  \geq 1 - (3\kappa +1) \, \omega_\Phi(\epsilon) - \left( \frac{r}{|\Lambda_p|} \right)^\frac{1}{2}
\]
and thus
\[
  2(3\kappa+1) \, \omega_\Phi(\epsilon) \geq 1 - \left( \frac{r}{|\Lambda_p|} \right)^\frac{1}{2}.
\]
\par 
Since $\Phi$ is uniformly continuous by \cite[Theorem 4.1]{Oza}, $\lim_{\epsilon \to 0} \omega_\Phi(\epsilon) = 0$ holds. Choosing $\epsilon > 0$ so small that $(3\kappa +1) \, \omega_\Phi(\epsilon) \leq \frac{1}{4}$, we get
\[
  \frac{1}{2} \geq 1 - \left( \frac{r}{|\Lambda_p|} \right)^\frac{1}{2}.
\]
for all $p \in \mathbb P$, which is impossible.
\end{proof}
\begin{remark}
It is crucial that $r \in \posints$ depend only on $\epsilon > 0$, but not on any particular $p$.
\end{remark}
\section{The non-amenability of $\ell^\infty({\cal K}(\ell^2 \oplus E))$}
It is possible to adapt the argument from \cite{Oza} to prove the non-amenability of $\ell^\infty({\cal K}(\ell^2))$. In this section, we shall go one step further: we shall see that, for certain Banach spaces $E$, the Banach algebra $\ell^\infty({\cal K}(\ell^2 \oplus E))$ is not amenable.
\par
Recall (from \cite{LT}, for instance) that a sequence $( x_n )_{n=1}^\infty$ in a Banach space $E$ is called a \emph{basis} for $E$ if, for each $x \in E$, there is a unique sequence $(\lambda_n)_{n=1}^\infty$ of scalars such that
\begin{equation} \label{basic}
  x= \sum_{n=1}^\infty \lambda_n x_n
\end{equation}
(with convergence in the norm topology of $E$). For instance, the canonical basis $( e_n )_{n=1}^\infty$ of $\ell^p$ is a basis in this sense if $p \in [1,\infty)$ (but not if $p = \infty$). For each $x \in E$ and $n \in \posints$, set $\langle x, x_n^\ast \rangle:= \lambda_n$, where $\lambda_n$ is the coefficient of $x_n$ in the expansion (\ref{basic}) of $x$. Then $( x_n^\ast )_{n=1}^\ast$ is a sequence in $E^\ast$. 
\par 
Let $E$ be a Banach space, and let ${\cal F}(E)$ denote the bounded finite rank operators on it. Identifying ${\cal F}(E)$ with the algebraic tensor product $E \tensor E^\ast$, we define the \emph{trace} on ${\cal F}(E)$ via
\[
  \Tr \!: {\cal F}(E) \to \comps, \quad x \tensor x^\ast \mapsto \langle x, x^\ast \rangle.
\]
\par 
We collect the following (mostly well known) facts for later use:
\begin{lemma} \label{tracelem}
Let $E$ be a Banach space. Then the following hold:
\begin{items}
\item $\Tr \, ST = \Tr \, TS \qquad (S \in {\cal F}(E), \, T \in {\cal B}(E))$;
\item if $F$ is another Banach space and $T = \left[ \begin{smallmatrix} T_{1,1} & T_{1,2} \\ T_{2,1} & T_{2,2} \end{smallmatrix} \right] \in {\cal F}(E \oplus F)$, then 
\[
  \Tr \, T = \Tr \, T_{1,1} + T_{2,2};
\]
\item if $E$ has a basis $(x_n )_{n=1}^\infty$, then 
\[
  \Tr \, T = \sum_{n=1}^\infty \langle T x_n, x_n^\ast \rangle \qquad (T \in {\cal F}(E)).
\]
\end{items}
\end{lemma}
\begin{proof}
(i) and (ii) are immediate from the definition (and well known).
\par 
For (iii), note that, by linearity, it suffices to check the claim for $T = x \tensor x^\ast$ with $x \in E$ and $x^\ast \in E^\ast$.
Since $( x_n )_{n=1}^\infty$ is a basis for $E$, there is a unique sequence $(\lambda_n)_{n=1}^\infty$ in $\comps$ such that
(\ref{basic}) holds. It follows that
\[
  \Tr \, T = \left\langle \sum_{n=1}^\infty \lambda_n x_n, x^\ast \right\rangle
           = \sum_{n=1}^\infty \lambda_n \langle x_n, x^\ast \rangle
           = \sum_{n=1}^\infty \langle x, x_n^\ast \rangle \langle x_n, x^\ast \rangle
           = \sum_{n=1}^\infty \langle T x_n, x_n^\ast \rangle
\]
as claimed.
\end{proof}
\begin{theorem} \label{nonamthm}
Let $E$ be a Banach space with a basis $( x_n )_{n=1}^\infty$ such that there is $C > 0$ with
\begin{equation} \label{condi}
  \sum_{n=1}^\infty \| S x_n \| \| T x_n^\ast \| \leq C \, N \| S \| \| T \|
  \qquad (N \in \posints, \, S \in {\cal B}(E,\ell^2_N), \, T \in {\cal B}(E^\ast, \ell^2_N)).
\end{equation}
Then $\ell^\infty({\cal K}(\ell^2 \oplus E))$ is not amenable.
\end{theorem}
\begin{proof}
We identify $\ell^\infty({\cal K}(\ell^2 \oplus E))$ with $\ell^\infty(\mathbb{P}, {\cal K}(\ell^2 \oplus E))$, which we will denote simply by $\A$. Each summand of $\A$ has a canonical block matrix structure
\begin{equation} \label{block}
  {\cal K}(\ell^2 \oplus E) = \begin{bmatrix} {\cal K}(\ell^2) & {\cal K}(E,\ell^2) \\
                              {\cal K}(\ell^2,E) & {\cal K}(E)
                              \end{bmatrix}.
\end{equation}
For each $p \in \mathbb P$, we embed ${\cal B}(\ell^2(\Lambda_p))$ into ${\cal K}(\ell^2 \oplus E)$ as upper left corner of the block matrix (\ref{block}). This allows us to consider $\text{$\ell^\infty$-}\bigoplus_{p \in \mathbb P} {\cal B}(\ell^2(\Lambda_p))$ s a closed subalgebra of $\A$. In particular, we can identify
\[
  F := \left\{ ( \pi_p(g_j) )_{p \in \mathbb P} : j =1, \ldots, m+1 \right\}
\]
with a finite subset of $\A$. Furthermore, we let $\A$ act (as block diagonal matrices) on the space
\[
  \ell^2(\mathbb{P}, \ell^2 \oplus E) \cong \ell^2(\mathbb{P},\ell^2) \oplus \ell^2(\mathbb{P},E)
\]
We shall identify $\ell^2(\mathbb{P},\ell^2)$ and $\ell^2(\mathbb{P},E)$ with completions of the algebraic tensor products
$\ell^2(\mathbb{P}) \tensor \ell^2$ and $\ell^2(\mathbb{P}) \tensor E$.
\par 
For any $p \in \mathbb P$, let $P_p \in {\cal B}(\ell^2)$ be the canonical projection onto the first $|\Lambda_p|$ coordinates of the $p^\mathrm{th}$ $\ell^2$-summand of $\ell^2(\mathbb{P},\ell^2) \oplus \ell^2(\mathbb{P},E)$. Let $a = ( a_p )_{p \in \mathbb P}$ and $b = ( b_p )_{p \in \mathbb P}$ be elements of $\A$. By \cite[Lemma 2.1(i)]{Oza}, we have
\begin{multline} \label{estim1}
  \sum_{q \in \mathbb P} \sum_{n=1}^\infty \| P_p a(e_q \tensor e_n) \| \| P_p^\ast b^\ast (e_q^\ast \tensor e_n^\ast) \|
  = \sum_{n=1}^\infty \| P_p a_p e_n \| \| P_p^\ast b_p^\ast e_n^\ast \| \\
  \leq |\Lambda_p| \| a_p \| \| b_p \| \leq |\Lambda_p| \| a \| \| b \|,
\end{multline}
and from (\ref{condi}), we conclude that
\begin{multline} \label{estim2}
  \sum_{q \in \mathbb P} \sum_{n=1}^\infty \| P_p a(e_q \tensor x_n) \| \| P_p^\ast b^\ast (e_q^\ast \tensor x_n^\ast) \| 
  = \sum_{n=1}^\infty \| P_p a_p x_n \| \| P_p^\ast b_p^\ast x_n^\ast \| \\
  \leq C \, |\Lambda_p| \| a_p \| \| b_p \| \leq C \, |\Lambda_p| \| a \| \| b \|,
\end{multline}
\par
Set $e = (P_p)_{p \in \mathbb P}$. Then $e$ is an idempotent in $\A$ with
\[
  e \A e =   \text{$\ell^\infty$-}\bigoplus_{p \in \mathbb P} {\cal B}(\ell^2(\Lambda_p)).
\]
\par 
Assume now towards a contradiction that $\ell^\infty(\mathbb{P}, {\cal K}(\ell^2 \oplus E))$ is amenable, and let $\epsilon > 0$ be arbitrary. By Lemma \ref{amlem} there are thus $a_1, b_1, \ldots, a_r , b_r \in \A$ such that (\ref{diag1}) holds and
\begin{equation} \label{diag3}
  \left\| \sum_{k=1}^r x a_k \tensor b_k - a_k \tensor b_k x \right\| < \frac{\epsilon}{(C+1)(m+1)} \qquad (x \in F)
\end{equation}
\par 
For $p,q \in \mathbb P$ and $n \in \posints$, define
\begin{multline*}
  T_p(q,n) := \sum_{k=1}^r P_p a_k (e_q \tensor e_n) \tensor P_p^\ast b_k^\ast (e_q^\ast \tensor e_n^\ast) +
                           P_p a_k (e_q \tensor x_n) \tensor P_p^\ast b_k^\ast (e_q^\ast \tensor x_n^\ast) \\
  \in \ell^2(\Lambda_p) \Tensor \ell^2(\Lambda_p).
\end{multline*}
For $j=1, \ldots, m+1$ and fixed $p \in \mathbb P$, we then have by (\ref{estim1}) and (\ref{estim2}):
\begin{multline*}
  \sum_{q \in \mathbb P} \sum_{n=1}^\infty \| T_p(q,n) - ((\pi_p(g_j) \tensor \pi_p(g_j))T_p(q,n) \| \\
  \leq (1 +C ) |\Lambda_p|
  \left\| \sum_{k=1}^r ( \pi_p (g_j) )_{p \in \mathbb P} a_k \tensor b_k - 
    a_k \tensor b_k ( \pi_p (g_j) )_{p \in \mathbb P} \right\| \leq \frac{\epsilon}{m+1} | \Lambda_p | 
\end{multline*}
and thus
\begin{equation} \label{ineq1}
  \sum_{q \in \mathbb P} \sum_{n=1}^\infty \sum_{j=1}^{m+1} \| T_p(q,n) - ((\pi_p(g_j) \tensor \pi_p(g_j))T_p(q,n) \| 
  \leq \epsilon |\Lambda_p|
\end{equation}
For $k=1, \ldots, r$, write $a_k = ( a_{k,q} )_{q \in \mathbb P}$ and $b_k = ( b_{k,q} )_{q \in \mathbb P}$, and note that
\begin{equation} \label{traceestim}
  \begin{split}
  \sum_{q \in \mathbb P} \sum_{n=1}^\infty \| T_p(q,n) \|  
  & \geq \sum_{n=1}^\infty\left|\sum_{k=1}^r \langle P_p a_{k,p} e_n, P_p^\ast b_{k,p}^\ast e_n^\ast \rangle +
                                             \langle P_p a_{k,p} x_n, P_p^\ast b_{k,p}^\ast x_n^\ast \rangle\right| \\
  & \geq\left| \sum_{n=1}^\infty \sum_{k=1}^r \langle P_p a_{k,p} e_n, P_p^\ast b_{k,p}^\ast e_n^\ast \rangle +
  \sum_{n=1}^\infty \sum_{k=1}^r \langle P_p a_{k,p} x_n, P_p^\ast b_{k,p}^\ast x_n^\ast \rangle\right| \\
  & = \left| \Tr \sum_{k=1}^r b_{k,p} P_p a_{k,p} \right|, \qquad\text{by Lemma \ref{tracelem}(ii) and (iii)}, \\
  & = \left| \Tr \sum_{k=1}^r P_p a_{k,p} b_{k,p} \right|, \qquad\text{by Lemma \ref{tracelem}(i)}, \\
  & = | \Tr \, P_p |, \qquad\text{because $\sum_{k=1}^r a_{k,p} b_{k,p} = P_p$}, \\
  & = | \Lambda_p |.
  \end{split}
\end{equation}
Of course, the inequalities (\ref{ineq1}) and (\ref{traceestim}) remain intact if we only take the sum over those $q \in \mathbb P$ and $n \in \posints$ such that $T_p(q,n) \neq 0$. Both inequalties together then imply that there is $q \in \mathbb P$ and $n \in \posints$ with $T_p(q,n) \neq 0$ such that
\[
  \sum_{j=1}^{m+1} \| T_p(q,n) - ((\pi_p(g_j) \tensor \pi_p(g_j))T_p(q,n) \|  \leq \epsilon \| T_p(q,n) \|
\]
and thus
\begin{equation} \label{ineq2}
  \| T_p(q,n) - ((\pi_p(g_j) \tensor \pi_p(g_j))T_p(q,n) \| \leq \epsilon \| T_p(q,n) \| \qquad  (j=1, \ldots, m+1).
\end{equation}
From the definition of $T_p(q,n)$ it is clear that there are $\xi_{1,p}, \eta_{1,p}, \ldots, \xi_{2r,p}, \eta_{2r,p} \in \ell^2(\Lambda_p)$ with $T_p(q,n) = \sum_{k=1}^{2r} \xi_{k,p} \tensor \eta_{k,p}$. In view of (\ref{ineq2}), we thus have $\xi_{1,p}, \eta_{1,p}, \ldots, \xi_{2r,p}, \eta_{2r,p} \in \ell^2(\Lambda_p)$ with $\sum_{k=1}^{2r} \xi_{k,p} \tensor \eta_{k,p} \neq 0$ such that 
\begin{multline*}
  \left\| \sum_{k=1}^{2r} \xi_{j,p} \tensor \eta_{k,p} - (\pi_p(g_j) \tensor \pi_p(g_j))(\xi_{k,p} \tensor \eta_{k,p}) 
  \right\|_{\ell^2(\Lambda_p) \Tensor \ell^2(\Lambda_p)} \\ \leq \epsilon \left\| \sum_{k=1}^{2r} \xi_{k,p} \tensor \eta_{k,p}
  \right\|_{\ell^2(\Lambda_p) \Tensor \ell^2(\Lambda_p)} \qquad (j=1, \ldots, m+1).
\end{multline*}
Since $p \in \mathbb P$ was arbitrary, this contradicts Ozawa's lemma.
\end{proof}
\section{The non-amenability of ${\cal B}(\ell^p(E))$}
With Theorem \ref{nonamthm} proven, we shall now see that the condition imposed on the basis of $E$ in that theorem is automatically satisfied by the canonical basis $(e_n)_{n=1}^\infty$ of $\ell^p$ with $p \in (1,\infty)$. (As is customary, we set $p' := \frac{p}{p-1}$, so that $(\ell^p)^\ast = \ell^{p'}$.).
\par 
Let $p \in [1,\infty)$, and let $E$ and $F$ be Banach spaces. Recall (from \cite{DJT}, for instance) that a linear operator $T \!: E \to F$ is called \emph{$p$-summing} if the amplification $\id_{\ell^p} \tensor T : \ell^p \tensor E \to \ell^p \tensor F$ extends to a bounded linear operator from $\ell^p \wTensor E$---$\wTensor$ stands for the injective Banach space tensor product---to $\ell^p(F)$. In this case, the operator norm of $\id_{\ell^p} \tensor T \!: \ell^p \wTensor E \to \ell^p(F)$ is called the \emph{$p$-summing norm} of $T$ and denoted by $\pi_p(T)$.
\begin{lemma} \label{estl}
Let $p \in (1,\infty)$. Then there is $C > 0$ such that
\[
  \sum_{n=1}^\infty \| S e_n \|_{\ell^2_N} \| T e^\ast_n \|_{\ell^2_N} \leq C \, N \| S \| \| T \| 
  \qquad (N \in \posints, \, S \in {\cal B}(\ell^p, \ell^2_N), \, T \in {\cal B}(\ell^{p'}, \ell^2_N)).
\]
\end{lemma}
\begin{proof}
We can, for each $N \in \posints$, algebraically identify ${\cal B}(\ell^p, \ell^2_N)$ and ${\cal B}(\ell^{p'}, \ell^2_N)$ with the algebraic tensor products $\ell^{p'} \tensor \ell^2_N$ and $\ell^p \tensor \ell^2_N$, respectively. On the other hand, those tensor products can also be identified with the spaces $\ell^{p'}(\ell^2_N)$ and $\ell^p(\ell^2_N)$, respectively. H\"olders inequality yields immediately that
\[
  \sum_{n=1}^\infty \| S e_n \|_{\ell^2_N} \| T e^\ast_n \|_{\ell^2_N} \leq \| S \|_{\ell^{p'}(\ell^2_N)} \| T \|_{\ell^{p}(\ell^2_N)}
  \qquad (N \in \posints, \, S \in {\cal B}(\ell^p, \ell^2_N), \, T \in {\cal B}(\ell^{p'}, \ell^2_N)).
\]
\par 
With $\wTensor$ denoting the injective tensor product, we have the isometric identifications
\[
  {\cal B}(\ell^p, \ell^2_N) \cong \ell^{p'} \wTensor \ell^2_N 
  \qquad\text{and}\qquad
  {\cal B}(\ell^{p'}, \ell^2_N) \cong \ell^p \wTensor \ell^2_N,
\]
so that
\begin{multline*}
  \| S \|_{\ell^{p'}(\ell^2_N)} \leq \pi_{p'}(\id_{\ell^2_N}) \| S \|
  \quad\text{and}\quad
  \| T \|_{\ell^{p}(\ell^2_N)} \leq \pi_p(\id_{\ell^2_N}) \| T \| \\
  \qquad (N \in \posints, \, S \in {\cal B}(\ell^p, \ell^2_N), \, T \in {\cal B}(\ell^{p'}, \ell^2_N)).
\end{multline*}
\par
By \cite[Theorem 5]{Gor} (compare also \cite[Theorem 10.3]{TJ})
\[
  \pi_{p'}(\id_{\ell^2_N}) \sim N^\frac{1}{2} \sim \pi_p(\id_{\ell^2_N})
\]
holds; in particular, there are $C_{p'}, C_p > 0$ with
\[
  \pi_{p'}(\id_{\ell^2_N}) \leq C_{p'} \sqrt{N} \quad\text{and}\quad \pi_p(\id_{\ell^2_N}) \leq C_p \sqrt{N} 
  \qquad (N \in \posints).
\]
Letting $C := C_{p'} C_p$ yields the desired constant.
\end{proof}
\par 
By Lemma \ref{estl}, $\ell^p$ therefore satisfies the hypothesis of Theorem \ref{nonamthm} for any $p \in (1,\infty)$, so that $\ell^\infty({\cal K}(\ell^2 \oplus \ell^p))$ is not amenable, but more is true. 
\par 
Recall the definition of an ${\cal L}^p$-space from \cite{LP}:
\begin{definition}
Let $p \in [1,\infty]$. 
\begin{alphitems}
\item A Banach space $E$ is called an \emph{${\cal L}^p_\lambda$-space} with $\lambda \geq 1$ if, for any finite-dimensional subspace $X$ of $E$, there are $N \in \posints$ and an $N$-dimensional subspace $Y$ of $E$ containing $X$ such that $d(Y,\ell^p_N) \leq \lambda$, where $d$ is the Banach--Mazur distance. 
\item We call $E$ an \emph{${\cal L}^p$-space} if it is an ${\cal L}^p_\lambda$-space for some $\lambda \geq 1$. 
\end{alphitems}
\end{definition}
\par 
All $L^p$-spaces are ${\cal L}^p$-spaces, but the converse fails: For $p \in (1,\infty) \setminus \{ 2 \}$, the space $\ell^2 \oplus \ell^p$ is an ${\cal L}^p$-space, but not an $L^p$-space (\cite[Example 8.2]{LP}). By \cite[Theorem 4.3]{DR}, all infinite-dimensional ${\cal L}^p$-spaces $E$ are equivalent when it comes to the question of whether $\ell^\infty({\cal K}(E))$ is amenable: either $\ell^\infty({\cal K}(E))$ is amenable for all infinite-dimensional ${\cal L}^p$-spaces $E$ or it is non-amenable for all such spaces.
\par 
In view of Theorem \ref{nonamthm} and Lemma \ref{estl}, we thus have:
\begin{proposition} \label{Lpprop}
Let $p \in (1,\infty)$, and let $E$ be an infinite-dimensional ${\cal L}^p$-space. Then $\ell^\infty({\cal K}(E))$ is not amenable.
\end{proposition}
\par 
We can finally deduce the main result of this paper:
\begin{theorem}
Let $p \in (1,\infty)$, and let $E$ be an ${\cal L}^p$-space. Then ${\cal B}(\ell^p(E))$ is not amenable.
\end{theorem}
\begin{proof}
Assume towards a contradiction that ${\cal B}(\ell^p(E))$ is amenable. By \cite[Theorem 2.1]{DR}, this forces $\ell^\infty({\cal B}(\ell^p(E)))$ to be amenable as well. Since $E$ is an ${\cal L}^p$-space, so is $\ell^p(E)$; in particular, $\ell^p(E)$ is reflexive and has the approximation property, so that ${\cal K}(\ell^p(E))$ has a bounded approximate identity. Consequently, $\ell^\infty({\cal K}(\ell^p(E)))$ is an ideal in $\ell^\infty({\cal B}(\ell^p(E)))$ with a bounded approximate identity and thus an amenable Banach algebra in its own right (\cite[Proposition 2.3.3]{LoA}), which is impossible by Proposition \ref{Lpprop}.
\end{proof}
\begin{corollary} \label{nonamcor}
Let $p \in (1,\infty)$, and let $E$ be an ${\cal L}^p$-space such that $E \cong \ell^p(E)$. Then ${\cal B}(E)$ is not amenable.
\end{corollary}
\par
For any $p \in (1,\infty)$, the spaces $E = \ell^p$ and $E= L^p[0,1]$ satisfy the hypothesis of Corollary \ref{nonamcor}, so that ${\cal B}(\ell^p)$ and ${\cal B}(L^p[0,1])$ are not amenable. Since any infinite-dimensional, separable $L^p$-space is isomorphic to either $\ell^p$ or $L^p[0,1]$ (\cite[p.\ 83]{Woj}), we even obtain:
\begin{corollary} 
Let $p \in (1,\infty)$, and let $E$ be an infinite-dimensional, separable $L^p$-space. Then ${\cal B}(E)$ is not amenable.
\end{corollary}
\begin{remark}
If one is only interested in the non-amenability of ${\cal B}(\ell^p)$ for $p \in (1,\infty)$, a more direct route is possible: According to \cite[p.\ 73]{LT}, the isomorphism 
\begin{equation} \label{lpsum}
   \ell^p \cong \text{$\ell^p$-}\bigoplus_{n=1}^\infty \ell_n^2
\end{equation}
holds. Using, for $q \in \mathbb P$, the projection from $\ell^p$ onto $\ell^2(\Lambda_q) = \ell^2_{|\Lambda_q|}$ according to (\ref{lpsum}), then invoking Lemma \ref{estl}, and finally following more or less \cite{Oza} will also yield a proof of the non-amenability of ${\cal B}(\ell^p)$. However, unlike the argument presented here, we do not see how this way of reasoning will yield the non-amenability of ${\cal B}(E)$ for every ${\cal L}^p$-space with $E \cong \ell^p(E)$, such as $L^p[0,1]$.
\end{remark}
\renewcommand{\baselinestretch}{1.0}
\renewcommand{\baselinestretch}{1.2}
\dated
\vfill
\begin{tabbing}
\textit{Author's address}: \= Department of Mathematical and Statistical Sciences \\
\> University of Alberta \\
\> Edmonton, Alberta \\
\> Canada T6G 2G1 \\[\medskipamount]
\textit{E-mail}: \> \texttt{vrunde@ualberta.ca} \\[\medskipamount]
\textit{URL}: \> \texttt{http://www.math.ualberta.ca/$^\sim$runde/}
\end{tabbing}           
\end{document}